\documentclass[12pt,a4paper]{amsart}
\usepackage{amsfonts,amssymb,amscd,amsmath,enumerate,verbatim}
\usepackage[latin1]{inputenc}
\usepackage{amscd}
\usepackage{cleveref}
\usepackage{multicol}
\usepackage{color}
\usepackage{tikz}
\input xy
\xyoption{all}

\def\opn#1#2{\def#1{\operatorname{#2}}} 
	\opn\chara{char} \opn\length{\ell} \opn\pd{pd} \opn\rk{rk}
	\opn\projdim{proj\,dim} \opn\injdim{inj\,dim} \opn\rank{rank}
	\opn\depth{depth} \opn\grade{grade} \opn\height{height}
	\opn\embdim{emb\,dim} \opn\codim{codim}
	\opn\Cl{Cl}
	
	\opn\Tr{Tr} \opn\bigrank{big\,rank}
	\opn\superheight{superheight}\opn\lcm{lcm}
	\opn\trdeg{tr\,deg}
	\opn\rdeg{rdeg}
	\opn\reg{reg} \opn\lreg{lreg} \opn\ini{in} \opn\lpd{lpd}
	\opn\size{size} \opn\sdepth{sdepth}
	\opn\link{link}\opn\fdepth{fdepth}\opn\lex{lex}
	\opn\tr{tr}
	\opn\type{type}
	\opn\gap{gap}
	\opn\arithdeg{arith-deg}
	\opn\revlex{revlex}
	\opn\div{div} \opn\Div{Div} \opn\cl{cl} \opn\Cl{Cl}
	\opn\Spec{Spec} \opn\Supp{Supp} \opn\supp{supp} \opn\Sing{Sing}
	\opn\Ass{Ass} \opn\Min{Min}\opn\Mon{Mon}
	\opn\Ann{Ann} \opn\Rad{Rad} \opn\Soc{Soc}
	\opn\Im{Im} \opn\Ker{Ker} \opn\Coker{Coker} \opn\Am{Am}
	\opn\Hom{Hom} \opn\Tor{Tor} \opn\Ext{Ext} \opn\End{End}
	\opn\Aut{Aut} \opn\id{id}
	
	\opn\nat{nat}
	\opn\pff{pf}
	\opn\Pf{Pf} \opn\GL{GL} \opn\SL{SL} \opn\mod{mod} \opn\ord{ord}
	\opn\Gin{Gin} \opn\Hilb{Hilb}\opn\sort{sort}
	\opn\PF{PF}\opn\Ap{Ap}
	\opn\mult{mult}
	\opn\bight{bight}
	\opn\div{div}
	\opn\Div{Div}
	\opn\aff{aff}
	\opn\relint{relint} \opn\st{st}
	\opn\lk{lk} \opn\cn{cn} \opn\core{core} \opn\vol{vol}  \opn\inp{inp} 
	\opn\nilpot{nilpot}
	\opn\link{link} \opn\star{star}\opn\lex{lex}\opn\set{set}
	\opn\width{wd}
	\opn\Fr{F}
	\opn\QF{QF}
	\opn\G{G}
	\opn\type{type}\opn\res{res}
	\opn\conv{conv}
	\opn\Int{Int}
	\opn\Deg{Deg}
	\opn\Sym{Sym}
	\opn\Con{Con}
	\opn\gr{gr}
	
	\def\pot#1#2{#1[\kern-0.28ex[#2]\kern-0.28ex]}

	\opn\dirlim{\underrightarrow{\lim}}
	\opn\inivlim{\underleftarrow{\lim}}
	\def\Implies{\ifmmode\Longrightarrow \else
		\unskip${}\Longrightarrow{}$\ignorespaces\fi}
	\def\implies{\ifmmode\Rightarrow \else
		\unskip${}\Rightarrow{}$\ignorespaces\fi}
	\def\iff{\ifmmode\Longleftrightarrow \else
		\unskip${}\Longleftrightarrow{}$\ignorespaces\fi}

	\let\:=\colon
	\newtheorem{Theorem}{Theorem}[section]
	
	\newtheorem{Corollary}[Theorem]{Corollary}

	\theoremstyle{definition}

	\newtheorem{Example}[Theorem]{Example}
	
	\let\epsilon\varepsilon
	\let\kappa=\varkappa
	\opn\dis{dis}
	\def\pnt{{\raise0.5mm\hbox{\large\bf.}}}
	
	\opn\Lex{Lex}
	
\begin{document}
\title{Vertex connectivity of chordal graphs}
\author[T\`ai Huy H\`a]{T\`ai Huy H\`a}
\address{Mathematics Department, Tulane University, 6823 St. Charles Avenue, New Orleans, LA 70118, USA}
\email{tha@tulane.edu}
\author[Takayuki Hibi]{Takayuki Hibi}
\address{Department of Pure and Applied Mathematics, Graduate School of Information Science and Technology, Osaka University, Suita, Osaka 565-0871, Japan}
\email{hibi@math.sci.osaka-u.ac.jp}
\dedicatory{}
\thanks{}
\begin{abstract}
Let $G$ be a finite graph and $\kappa(G)$ the vertex connectivity of $G$.  A chordal graph $G$ is called chordal$^*$ if no vertex of $G$ is adjacent to all other vertices of $G$.  Using the syzygy theory in commutative algebra, it is proved that every chordal$^*$ graph $G$ on $n$ vertices satisfies $\kappa(G) \leq (n - 1) - \lceil2\sqrt{n}-2\,\rceil$.  Furthermore, given an integer $0 \leq \kappa \leq (n - 1) - \lceil2\sqrt{n}-2\,\rceil$, a chordal$^*$ graph $G$ on $n$ vertices satisfying $\kappa(G) = \kappa$ is constructed. 
\end{abstract}	
\maketitle
\thispagestyle{empty}

\section*{Introduction}
Let $G$ be a finite {\em simple} graph, i.e., $G$ does not contain loops nor multiple edges. Let $V(G)$ and $E(G)$ denote the vertex and edge sets of $G$, respectively. Let $\kappa(G)$ denote the {\em vertex connectivity} of $G$. That is, $\kappa(G)$ is the minimum cardinality of a subset $W \subset V(G)$ for which the induced subgraph $G|_{V(G) \setminus W}$ is disconnected. The present paper studies the {\em spectrum} of the  vertex connectivity of $G$, when $G$ is a {\em chordal} graph. Precisely, we investigate the set of all possible values that $\kappa(G)$ may take when $G$ is a chordal graph on $n$ vertices, for a fixed positive integer $n$. This is inspired by the authors' recent study \cite{MAXMIN} of resolution of edge ideals of finite graphs.

A vertex $x \in V(G)$ is called {\em universal} if $\{x, y\} \in E(G)$ for all $y \in V(G)$ with $y \neq x$.  When discussing the vertex connectivity of $G$, one can assume that $G$ has no universal vertex.  In fact, if $x \in V(G)$ is universal, then $\kappa(G) = \kappa(G|_{V(G) \setminus \{x\}})+1$. 

A {\em vertex cover} of $G$ is a subset $C \subset V(G)$ for which $e \cap C \neq \emptyset$ for all $e \in E(G)$.  Let $\tau_{\max}(G)$ denote the maximum size of minimal vertex covers of $G$.   It is shown \cite[Theorem 2.2]{CHLN} that $\tau_{\max}(G) \geq \lceil2\sqrt{n} - 2\rceil$, where $n = |V(G)|$.  

In combinatorics on finite graphs, invariants $\kappa(G)$ and $\tau_{\max}(G)$ appear to be fundamentally distinct.  However, our work shows that, in certain restricted settings, the syzygy theory of monomial ideals in commutative algebra establishes an explicit relationship between these invariants.  

A chordal graph $G$ is called chordal$^*$ if $G$ has no universal vertex.  When $G$ is a chordal$^*$ graph on $n$ vertices, in Theorem \ref{main}, we deduce from the aforementioned bound $\tau_{\max}(G^c) \geq \lceil2\sqrt{n} - 2\rceil$, where $G^c$ is the complementary graph of $G$, that 
\begin{eqnarray}
\label{kappa}
0 \leq \kappa(G) \leq (n - 1) - \lceil2\sqrt{n}-2\,\rceil.
\end{eqnarray}  
Furthermore, given an integer $0 \leq \kappa \leq (n - 1) - \lceil2\sqrt{n}-2\,\rceil$, we construct a chordal$^*$ graph $G$ on $n$ vertices satisfying $\kappa(G) = \kappa$. 

For a more general graph on $n$ vertices containing no universal vertices, the vertex connectivity could be bigger than $(n - 1) - \lceil2\sqrt{n}-2\,\rceil$.  For example, if $n$ is even and if $G_n$ is a finite graph obtained by removing the edges $\{1,2\}, \{3,4\}, \ldots, \{n-1,n\}$ from the complete graph $K_n$ on $\{1,\ldots,n\}$, then $\kappa(G_n) = n - 2$.  If $n$ is odd and if $H_{n}$ is a finite graph by adding the edges $\{2,n\}, \{3,n\}, \ldots, \{n-1,n\}$ to $G_{n-1}$, then $\kappa(H_n) = n - 3$. Thus, our bound (\ref{kappa}) is special for chordal graphs.

Our proof of (\ref{kappa}) is heavily based on the syzygy theory of monomial ideals in commutative algebra.  Our technique has its origin in \cite{H, TH}, where the connectivity of comparability graphs of distributive lattices and of modular lattices is studied.  Furthermore, in \cite{TH}, a short proof of Barnette's theorem that the connectivity of the graph of a simplicial $(d-1)$-sphere with $ d \geq 2$ is at least $d$ was given.


\section{Connectivity of chordal$^*$ graphs}

The background of commutative algebra required in the proof of Theorem \ref{main} can be found in \cite[Chapter 9]{HHgtm260}.

\begin{Theorem}
\label{main}
Let $G$ be a chordal$\,^*$ graph on $n$ vertices and $G^c$ the complementary graph of $G$.  One has 
\begin{eqnarray}
\label{kappatau}    
\kappa(G) + \tau_{\max}(G^c) \le n-1
\end{eqnarray}
and
$$0 \leq \kappa(G) \leq (n - 1) - \lceil2\sqrt{n}-2\,\rceil.$$  
Furthermore, given an integer $0 \leq \kappa \leq (n - 1) - \lceil2\sqrt{n}-2\,\rceil$, there exists a chordal$\,^*$ graph $G$ on $n$ vertices satisfying $\kappa(G) = \kappa$.
\end{Theorem}

\begin{proof}
Let $\Delta(G)$ denote the clique complex (\cite[p.~155]{HHgtm260}) of $G$.  Let $S=K[x_1,\ldots, x_n]$ denote the polynomial ring in $n$ variables over a field $K$ and $I(G^c)$ the edge ideal (\cite[p.~156]{HHgtm260}) of $G^c$.  The Stanley--Reisner ring $S/I_{\Delta(G)}$ (\cite[p.~132]{HHgtm260}) of $\Delta(G)$ coincides with the quotient ring $S/I(G^c)$.  Since $G$ is chordal$^*$, it follows that $G^c$ possesses no isolated vertex.  Hence $\projdim(S/I(G^c)) \geq \tau_{\max}(G^c) \geq \lceil2\sqrt{n}-2\,\rceil$ (\cite[Corollary 4.2]{MAXMIN}).  Furthermore, since $S/I(G^c)$ has linear resolution (\cite[Theorem 9.2.3]{HHgtm260}), it follows that $\projdim(S/I(G^c))$ is equal to the biggest integer $i$ for which $\beta_{i,i+1}(S/I(G^c)) \neq 0$.  Now, \cite[Lemma 2.1]{TH} guarantees that $\kappa(G)$ is the biggest integer $i$ with $\beta_{n-i,n-i+1}(S/I(G^c)) = 0$.  Hence, 
\begin{eqnarray}
\label{proj}
n - \kappa(G) = \projdim(S/I(G^c)) + 1 \geq \tau_{\max}(G^c) + 1 \geq \lceil2\sqrt{n}-2\,\rceil + 1,
\end{eqnarray} 
as desired.

In \cite[Theorem 5.5]{MAXMIN}, given an integer $\lceil2\sqrt{n}-2\,\rceil \leq j \leq n-1$, a gapfree chordal graph $\Gamma_j$ on $n$ vertices satisfying $\projdim(S/I(\Gamma_j)) = j$ is constructed.  (Recall that a {\em gapfree} graph is a connected graph $G$ for which if $e$ and $e'$ are edges of $G$ with $e \cap e' = \emptyset$, then there is an edge $f$ of $G$ with $e \cap f \neq \emptyset$ and $e' \cap f \neq \emptyset$.)  It follows from \cite[Corollary 6.9]{HV} that $(\Gamma_j)^c$ is a chordal$^*$ graph and, by (\ref{proj}), one has $\kappa((\Gamma_j)^c) = (n - 1) - j$.  Thus, given an integer $0 \leq \kappa \leq (n - 1) - \lceil2\sqrt{n}-2\,\rceil$, the chordal$\,^*$ graph $G=(\Gamma_{(n - 1) - \kappa})^c$ on $n$ vertices satisfies $\kappa(G) = \kappa$.  
\end{proof}

It is unclear if the relationship (\ref{kappatau}) between $\tau_{\max}(G^c)$ and $\kappa(G)$ for a chordal$^*$ graph $G$ appears in the literature on finite graphs.

\begin{Example}
Let $C_n$ denote the cycle of length $n \geq 4$.  Since $\kappa(C_n) = 2$ and $\tau_{\max}((C_n)^c) = n-2$, one has $\kappa(C_n) + \tau_{\max}((C_n)^c) = n > n - 1$.   
\end{Example}



On the other hand, the {\em algebraic connectivity} of $G$, denoted by $a(G)$, was first introduced by Fiedler in \cite{F73} and has been much studied. A classical result of Fiedler \cite[4.1]{F73} states that if $G$ is not a complete graph, then $a(G) \le \kappa(G)$. Thus, as an immediate consequence of Theorem \ref{main}, we obtain the following corollary.

\begin{Corollary}
    \label{cor.alg}
    Let $G$ be a chordal$\,^*$ graph on $n$ vertices. Then,
    $$0 \le a(G) \le (n-1) - \lceil 2\sqrt{n} - 2 \rceil.$$
\end{Corollary}

It would be of interest to find all possible $0 \le a \le (n-1) - \lceil 2\sqrt{n}-2 \rceil$ which can be the algebraic connectivity of a chordal$^*$ graph on $n$ vertices.


\section{Classification}
It is a reasonable problem to classify all chordal$^*$ graphs $G$ on $n$ vertices satisfying $\kappa(G) = (n - 1) - \lceil2\sqrt{n}-2\,\rceil$. 

\begin{Example}
Let $n=6$.  Then $(n - 1) - \lceil2\sqrt{n}-2\,\rceil = 2$.  Each chordal$^*$ graph $G$ of Figure $1$ satisfies $\kappa(G) = 2$.  
\end{Example}
 
\begin{figure}[h]
\centering
\begin{tikzpicture}[scale=1.2]
\coordinate (z) at (0,1){};
\coordinate (p) at (-1,0.3){};
\coordinate (q) at (1,0.3){};
\coordinate (x) at (0,0){};
\coordinate (a) at (-0.7,-0.7){};
\coordinate (b) at (0.7,-0.7){};
\fill(z)circle(0.7mm);
\fill(p)circle(0.7mm);
\fill(q)circle(0.7mm);
\fill(x)circle(0.7mm);
\fill(a)circle(0.7mm);
\fill(b)circle(0.7mm);
\draw(z)--(p)--(a)--(b)--(q)--(x)--(b);
\draw(x)--(a);
\draw(x)--(z);
\draw(z)--(b);
\draw(z)--(a);
\end{tikzpicture}
\hspace{2cm}
\begin{tikzpicture}[scale=1.05]
\coordinate (a) at (1,1.5){};
\coordinate (b) at (0,1){};
\coordinate (c) at (0,0){};
\coordinate (d) at (1,-0.5){};
\coordinate (e) at (2,0){};
\coordinate (f) at (2,1){};
\fill(a)circle(0.7mm);
\fill(b)circle(0.7mm);
\fill(c)circle(0.7mm);
\fill(d)circle(0.7mm);
\fill(e)circle(0.7mm);
\fill(f)circle(0.7mm);
\draw(a)--(b)--(c)--(d)--(e)--(f)--cycle;
\draw(a)--(c)--(e)--cycle;
\end{tikzpicture}
\caption{Chordal$^*$ graphs with $6$ vertices}
\end{figure}
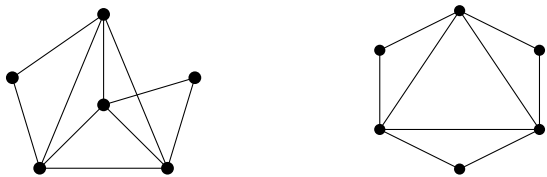

In particular, given an integer $n \geq 2$, a chordal$^*$ graphs $G$ on $n$ vertices satisfying $\kappa(G) = (n - 1) - \lceil2\sqrt{n}-2\,\rceil$ is, in general, not unique.  However,        

\begin{Theorem}
\label{classification}
Let $n \geq 4$ be a perfect square.  Then a chordal\,$^*$ graph on $n$ vertices satisfying $\kappa(G) = (n - 1) - \lceil2\sqrt{n}-2\,\rceil$ exists uniquely.  
\end{Theorem}

\begin{proof}
The existence follows from Theorem \ref{main}.  A chordal$^*$ graph on $4$ vertices satisfying $\kappa(G) = 1$ is only the path graph on $4$ vertices.  Suppose that a chordal$^*$ graph $G$ on $n \geq 9$ vertices satisfies $\kappa(G) = (n - 1) - \lceil2\sqrt{n}-2\,\rceil$.  It then follows from (\ref{proj}), that $\projdim(S/I(G^c)) = \tau_{\max}(G) = \lceil2\sqrt{n}-2\,\rceil$.  Now, it is shown \cite[Theorem 5.4]{MAXMIN} that a finite graph $H$ on $n$ vertices satisfying  $\tau_{\max}(H) = \lceil2\sqrt{n}-2\,\rceil$ exists uniquely.   
\end{proof}

\end{document}